\documentclass[a4paper,12pt]{amsart}
\usepackage{amssymb}
\usepackage{amsmath}
\usepackage{stmaryrd}

\def\a{\alpha}

\def\e{\varepsilon}

\def\vs{\vskip .6cm}

\def\beq{\begin{equation}}
\def\eeq{\end{equation}}
\def\bea{\begin{eqnarray*}}
\def\eea{\end{eqnarray*}}
\def\beaa{\begin{eqnarray}}
\def\eeaa{\end{eqnarray}}
\def\ba{\begin{array}}
\def\ea{\end{array}}

\def\e{\varepsilon}

\def \RM{\mathbb{R}}

\def \ZM{\mathbb{Z}}
\def \CM{\mathbb{C}}

\def \HM{\mathbb{H}}
\def \SM{\mathbb{S}}

\def\rk{\mathrm{rk}}

\def\be{\begin{equation}}
\def\ee{\end{equation}}

\def\rk{\mathrm{rk}}

\def\spin{\mathfrak{spin}}

\def\gg{\mathfrak{g}}
\def\hh{\mathfrak{h}}

\def\nn{\mathfrak{n}}

\def\mm{\mathfrak{m}}

\def\SU{\mathrm{SU}}
\def\U{\mathrm{U}}
\def\A{\mathrm{A}}
\def\B{\mathrm{B}}
\def\C{\mathrm{C}}
\def\D{\mathrm{D}}
\def\E{\mathrm{E}}
\def\G{\mathrm{G}}

\def\F{\mathrm{F}}
\def\SO{\mathrm{SO}}

\def\Sp{\mathrm{Sp}}
\def\Spin{\mathrm{Spin}}

\def\ad{\mathrm{ad}}

\def\Gr{\mathrm{Gr}}

\def\T{\mathrm{T}}

\def\RR{{\mathcal R}}
\def\W{{\mathcal W}}

\newtheorem{epr}{Proposition}[section]
\newtheorem{ath}[epr]{Theorem}
\newtheorem{elem}[epr]{Lemma}
\newtheorem{ecor}[epr]{Corollary}

\theoremstyle{definition}

\newtheorem{ere}[epr]{Remark}
\newtheorem{exe}[epr]{Example}

\title[Weakly complex structures]{Weakly complex homogeneous spaces}

\author{Andrei Moroianu, Uwe Semmelmann}

\address{Andrei Moroianu \\ CMLS\\ {\'E}cole
  Polytechnique \\ UMR 7640 du CNRS
\\ 91128 Palaiseau \\ France}
\email{am@math.polytechnique.fr}

\address{Uwe Semmelmann\\
Institut f\"ur Geometrie und Topologie \\
Fachbereich Mathematik\\
Universit{\"a}t Stuttgart\\
Pfaffenwaldring 57 \\
70569 Stuttgart, Germany
}
\email{uwe.semmelmann@mathematik.uni-stuttgart.de}

\date{\today}
\thanks{This work was supported through the program "Research in Pairs" by the Mathematisches Forschungsinstitut Oberwolfach. We thank the institute for the hospitality and the stimulating research environment. }

\begin{document}

\begin{abstract}

We complete our recent classification \cite{gms} of compact inner symmetric spaces with weakly complex tangent bundle
by filling up a case which was left open, and extend this classification
to the larger category of compact homogeneous spaces with positive Euler characteristic. We show that 
a simply connected compact equal rank homogeneous space has weakly complex tangent bundle if and only
if it is a product of compact equal rank homogeneous spaces which either carry an invariant almost complex structure
(and are classified by Hermann \cite{h55}), or have stably trivial tangent bundle (and are classified by
Singhof and Wemmer \cite{sw86}), or belong to an explicit list of weakly complex spaces which have neither 
stably trivial tangent bundle, nor carry invariant almost complex structures. 
\vs


\end{abstract}

\maketitle

\section{Introduction}

It is well-known \cite{a69} that a compact homogeneous space $G/H$ has non-vanishing Euler characteristic 
if and only if $G$ and $H$ have equal rank.
If this happens, then the Euler characteristic of $G/H$ is positive, equal to the quotient
of the cardinals of the Weyl groups: $\chi(G/H)=\sharp\W(G)/\sharp\W(H)$. For this reason, we will refer throughout this paper 
to compact homogeneous spaces with positive Euler characteristic as {\em equal rank homogeneous spaces}, a terminology
which seems to be used by some authors.

In this paper we study the following question: {\em Which equal rank homogeneous spaces have complex,
or more generally, weakly complex tangent bundle?} Recall that a real vector bundle $\tau$ is called weakly complex if 
there exists some trivial bundle $\epsilon$ such that $\tau\oplus\epsilon$ has a complex structure, that is, an endomorphism field
squaring to minus the identity. Note that no invariance property is required for the (weakly) complex structure in the above 
question.

Equal rank compact homogeneous spaces carrying {\em invariant} (also called {\em homogeneous}) 
almost complex structures were classified by Hermann \cite[Thm. 5.3]{h55}.
The classification is first reduced to the case where the group $G$ is simple and simply connected. Once this is done,
the group $H$ can either be semi-simple, which leads to nine cases, each of them corresponding to an exceptional group $G$, or
non semi-simple. In the latter situation, $H$ has to be the centralizer of a torus in $G$ up to four exceptional cases 
(one of which seems to have been overlooked in \cite{h55}).
Note that if $H$ is the centralizer of a torus, then the homogeneous space $G/H$ has an invariant {\em integrable} complex structure
\cite[Sect. 7]{wang}. This case includes the Hermitian symmetric spaces and generalized flag manifolds \cite[Ch. 8]{besse}.

In order to attack the general question, one needs completely different methods. The most powerful is a combination of
the Atyiah-Singer index theorem applied to some twisted Dirac operators, the Borel--Weil--Bott theorem and Weyl's dimension formula, which 
we recently used \cite{gms} in order to prove that the only compact irreducible inner
symmetric spaces with weakly complex tangent bundle are the even-dimensional spheres, the Hermitian symmetric spaces and (conceivably) the exceptional space $\E_7/(\SU(8)/\ZM_2)$.
As a mater of fact, the first important achievement of the present paper is to rule out this exceptional case (Theorem \ref{e7} below), thus completing the
classification in \cite{gms}.

One important ingredient which allows the passage from symmetric spaces to more general homogeneous spaces 
is the Borel--de Siebenthal \cite{bs49} classification of maximal subgroups of
maximal rank in compact simple Lie groups. It turns out that if $H$ is maximal in $G$ and $\rk(H)=\rk(G)$, then either 
$(G,H)$ is a symmetric pair, or it belongs to a list of seven exceptional cases, in each of them $G/H$ carrying an invariant almost
complex structure. 

The crucial assumption $\rk(H)=\rk(G)$ allows one to reduce the problem of the existence
of weakly complex structures on $G/H$ to the case where $G$ is simple.
Then, using the Borel--de Siebenthal classification, and the results in \cite{gms}, 
we prove the following classification result:

\begin{ath} An equal rank simply connected compact homogeneous space has weakly complex tangent bundle if and only if 
it is a product of manifolds belonging to the following list:
\begin{enumerate}
\item equal rank homogeneous spaces with an invariant almost complex structure;
\item equal rank homogeneous spaces with stably trivial tangent bundle;
\item one of the homogeneous spaces
\begin{itemize}
\item $\F_4/(\Spin(4)\times \T^2)$ 
\item $\F_4/(\Spin(4)\times \U(2))$ 
\item $\SO(2p+2q+1)/(\SO(2p)\times U)$
\item $\Sp(p+q)/(\Sp(1)^p\times U)$,
\end{itemize}
where  in the last two cases, $U$ is a rank $q$ subgroup of $\U(q)$ for some $q\ge 1$ and $p\ge 2$.
\end{enumerate}
Moreover, the manifolds in $(3)$ have neither stably trivial tangent bundles, nor invariant almost complex structures.
\end{ath}
The spaces in (1) have been classified by Hermann \cite[Thm. 5.3]{h55} up to one forgotten exceptional case 
$\E_8/(\A_5\times \A_2\times\T^1)$. The spaces in (2) were classified by
Singhof and Wemmer \cite[p. 159]{sw86}.

The precise statements are given in Theorems \ref{t1} and \ref{t2} below.

\noindent {\sc Acknowledgments}. We would like to thank Paul Gauduchon for many enlightening discussions
and for his interest in this work.

\vs 

\section{Preliminaries on compact Lie groups}

We will use throughout the text the standard notation for the compact simple Lie groups. By Cartan's classification there are four series for $n\ge 1$:
$$\A_n:=\SU(n+1),\qquad \B_n:=\Spin(2n+1),\qquad \C_n:=\Sp(n),\qquad\D_n:=\Spin(2n)$$
and five exceptional groups
$$\G_2,\ \F_4,\ \E_6,\ \E_7,\ \E_8,$$
where the subscript always indicates the rank. The attentive reader has already noticed that 
in the above list the $\D$ series should start at $n=3$ since $\D_1=\U(1)=\T^1$ and $\D_2=\A^1\times\A^1$
are not simple. By convention we take $\A_0\equiv \B_0\equiv\{1\}$ and we note the
exceptional isomorphisms $\C_1\equiv\B_1\equiv\A_1$, $\C_2\equiv\B_2$ and
$\A_3\equiv\D_3$.  

Recall first the following classical result which
describes the subgroups of maximal rank of a product of compact Lie groups:

\begin{elem}[\cite{bs49}]\label{l1} 
Let  a compact connected Lie group  $G$ be  the direct product of subgroups $G_1$ and $G_2$. If $L$
is a closed, connected subgroup of $G$ containing a maximal torus of $G$, then $L$ is the direct product of
$G_1\cap L$ and $G_2\cap L$.
\end{elem}

\begin{ecor}\label{c21}
If $T$ is a torus and $G$ a compact Lie group, then every connected subgroup $L$ of $T\times G$
with $\rk(L)=\rk(T\times G)$ is
of the form $T\times H$ where $H$ is a closed subgroup of $G$.
\end{ecor}

\begin{proof} From Lemma \ref{l1} we must have $L=(L\cap T)\times (L\cap G)$.
Moreover $\rk(L)=\rk(T\times G)=\rk(T)+\rk(G)\ge \rk(L\cap T)+\rk(L\cap G)=\rk(L)$, 
so in particular $\rk(T)=\rk(L\cap T)$. On the other hand, the torus $T$ has no proper subgroup of the same rank,
showing that $L\cap T=T$.
\end{proof}

In this paper we will consider simply connected compact homogeneous spaces $M$ of non-vanishing
Euler characteristic: $\chi(M)\neq 0$. This is equivalent to the existence
of compact Lie groups $H\subset G$ with $M=G/H$ and $\rk(H)=\rk(G)$. 
Note that there are in general several pairs $(G,H)$ representing $M$, but using Lemma \ref{l1} 
one can show that there exists
a pair $(G,H)$ representing $M$, with $G$ simply connected (and thus semi-simple) and $H$ connected.
Indeed, since $G$ is compact, 
it has a finite cover $\tilde G$ which is the direct product $\tilde G=T\times G'$ where $T$ is a torus and $G'$ is
simply connected. If $\tilde H$ denotes the inverse image of $H$ in 
$\tilde G$, one has $M=\tilde G/\tilde H$. 
The exact homotopy sequence
$$0=\pi_1(M)\to\pi_0(\tilde H)\to\pi_0(\tilde G)=0$$
shows that $\tilde H$ is connected.
By Corollary \ref{c21}, the subgroup $\tilde H\subset T\times G'$ can be written $\tilde H=T\times H'$, where $H':= (\tilde H\cap G')$. 
We thus can write $M=G'/H'$ with $G'$ simply connected and $H'$ connected, as claimed. 

\begin{ere}\label{rem} The notation $M=G/H$ makes sense when the embedding of $H$ in $G$ is 
specified. More generally, if $\rho:H\to G$ is a given morphism inducing a Lie algebra embedding 
$\rho_*:\hh\hookrightarrow\gg$, we denote by a slight abuse of
notation $G/\rho(H)$ by $G/H$. The justification of this notation is that the space $G/\rho(H)$ is uniquely defined by Lie algebra embedding $\hh\hookrightarrow\gg$, 
which in most cases is fixed by the context, so there is no risk of confusion. 
The main advantage is that we do not have to compute the (discrete) kernel of $\rho$ explicitly, which is a tough task in general. 
\end{ere}
\begin{exe}
The embedding $\spin(16)\hookrightarrow\mathfrak{e}_8$ induces
a group morphism $\Spin(16) \to\E_8$ whose kernel is $\ZM_2$, generated by the volume element 
of $\Spin(16)$ ({\em cf.} \cite[Thm. 6.1]{a96}). 
With the convention from the previous remark, we will thus denote the corresponding homogeneous 
space $\mathrm{E}_8/(\mathrm{Spin}(16)/\mathbb{Z}_2)$ simply by $\E_8/\D_8$.
\end{exe}

As already mentioned in the introduction, our study leans heavily on the Borel--de Siebenthal classification of maximal proper subgroups of
simple Lie groups which we now recall:

\begin{epr}[\cite{bs49}]\label{p1}
Let $G$ be compact connected simple Lie group and let $L$ be a maximal proper connected subgroup with $\rk(L)=\rk(G)$. 
Then $G/L$ is either an irreducible inner symmetric space
or belongs to the following list:
\beq\label{borel}
\G_2/\A_2,\; \F_4/(\A_2)^2,  \; \E_6/ (\A_2)^3, \;  \E_7/(\A_5\times \A_2),  \;  \E_8/\A_8,  \;  \E_8/(\E_6\times \A_2), \; 
\E_8/(\A_4\times \A_4) .
\eeq
Each of these seven exceptional spaces admits an invariant almost complex structure.
\end{epr}

Using the classification of irreducible inner symmetric spaces (\cite[pp. 312--314]{besse}) and Proposition \ref{p1} we immediately get (keeping in mind
Remark \ref{rem}):

\begin{ecor}\label{c1} Let $H$ be a rank $n$ maximal proper subgroup of $G$. 

a) If $G=\A_n$, $H$ is conjugate to some subgroup $\A_p\times \A_q\times \T^1$ with $p+q=n-1$ and $p,q\ge 0$ (recall that $\A_0=\{1\}$
by convention).

b) If $G=\B_n$, $H$ is conjugate to some subgroup $\B_p\times \D_q$ with $p+q=n$, $p\ge 0$ and $q\ge 1$, diagonally embedded in $\B_n$.

c) If $G=\C_n$, $H$ is either conjugate to  $\U(n)=\A_{n-1}\times \T^1$ or to some subgroup $\C_p\times \C_q$ with $p+q=n$, and $p,q\ge 1$, diagonally embedded in $\C_n$.

d) If $G=\D_n$, $H$ is conjugate to either conjugate to  $\U(n)=\A_{n-1}\times \T^1$ or to some subgroup $\D_p\times \D_q$ with $p+q=n$, and $p,q\ge 1$, diagonally embedded in $\D_n$.
\end{ecor}

For later use, let us note the following consequence of Lemma \ref{l1} and Corollary \ref{c1}:

\begin{elem}\label{l3}
The groups $\A_{n_1}\times\ldots\times \A_{n_k}$ contain no semi-simple proper subgroups of rank $n_1+\ldots+n_k$.
\end{elem}
\begin{proof}
From Lemma \ref{l1}, it suffices to prove the statement for $k=1$.
Assume for a contradiction that $H\subset \A_n$ is a proper semi-simple subgroup of rank $n$, and let $G\subset \A_n$ be a maximal 
proper subgroup of $\A_n$ containing $H$. By Lemma \ref{l1} again, $G$ is semi-simple, thus contradicting Corollary \ref{c1} a).
\end{proof}

An easy induction using Corollary \ref{c1} shows  that we can actually describe all closed subgroups of maximal rank of the classical groups. 
It turns out that for the $\A$ series it is more convenient to state the result for the subgroups of $\U(n)$ rather than for those of $\A_n$:

\begin{elem}\label{class1}
Let $H$ be a compact Lie group of rank $n$.

a) If $H$ is a subgroup of $\U(n)$, there exist $k\ge 1$ integers $n_i\ge 1$ with $n_1+\ldots+n_k=n$
such that $H$ is conjugate to $\U({n_1})\times\ldots\times \U({n_k})$, diagonally embedded in $\U(n)$. 

b) If $H$ is a subgroup of $\B_n$, there exist integers $m,k,l\ge 0$, $p_i\ge 1$ for $1\le i\le k$,  $q_i\ge 1$ for $1\le i\le l$,
 with $m+p_1+\ldots+p_k+q_1+\ldots+q_l=n$
such that $H$ is conjugate to $\B_{m}\times\D_{p_1}\times\ldots\times \D_{p_k}\times \U(q_1)\times\ldots\times \U(q_l)$ 
embedded in $\B_n$ as
$$\B_{m}\times\D_{p_1}\times\ldots\times \D_{p_k}\times \U(q_1)\times\ldots\times \U(q_l)\subset \B_{m}\times\D_{p_1}\times\ldots\times \D_{p_k}\times \D_{q_1}\times\ldots\times \D_{q_l}\subset \B_n$$
(each unitary group is embedded in the corresponding special orthogonal group via the standard embedding $\U(q_i)\subset \SO(2q_i)$ and
the product is diagonally embedded in $\B_n$).

c) If $H$ is a subgroup of $\C_n$, there exist integers $k,l\ge 0$, $p_i\ge 1$ for $1\le i\le k$,  $q_i\ge 1$ for $1\le i\le l$,
 with $p_1+\ldots+p_k+q_1+\ldots+q_l=n$
such that $H$ is conjugate to $\C_{p_1}\times\ldots\times \C_{p_k}\times \U(q_1)\times\ldots\times \U(q_l)$ 
embedded in $\C_n$ as
$$\C_{p_1}\times\ldots\times \C_{p_k}\times \U(q_1)\times\ldots\times \U(q_l)\subset \C_{p_1}\times\ldots\times \C_{p_k}\times \C_{q_1}\times\ldots\times \C_{q_l}\subset \C_n$$
(each unitary group is embedded in the corresponding symplectic group via the standard embedding $\U(q_i)\subset \Sp(q_i)$ and
the product is diagonally embedded in $\C_n$).

d) If $H$ is a subgroup of $\D_n$, there exist integers $k,l\ge 0$, $p_i\ge 1$ for $1\le i\le k$,  $q_i\ge 1$ for $1\le i\le l$,
 with $p_1+\ldots+p_k+q_1+\ldots+q_l=n$
such that $H$ is conjugate to $\D_{p_1}\times\ldots\times \D_{p_k}\times \U(q_1)\times\ldots\times \U(q_l)$ 
embedded in $\D_n$ as
$$\D_{p_1}\times\ldots\times \D_{p_k}\times \U(q_1)\times\ldots\times \U(q_l)\subset \D_{p_1}\times\ldots\times \D_{p_k}\times \D_{q_1}\times\ldots\times \D_{q_l}\subset \B_n$$
(each unitary group is embedded in the corresponding special orthogonal group via the standard embedding $\U(q_i)\subset \SO(2q_i)$ and
the product is diagonally embedded in $\D_n$).
\end{elem}

\vs

\section{Weakly complex structures on homogeneous spaces}

We are now ready to attack our main problem: the classification of simply connected compact equal rank homogeneous 
spaces whose tangent bundle is weakly complex. 

Since we are interested in almost complex structures, it is perhaps an appropriate place to
recall that if $H$ is the centralizer of a torus in $G$, then $\rk(H)=\rk(G)$ and 
$G/H$ automatically carries an invariant complex structure (see {\em e.g.} \cite[Sect. 7]{wang}). 
We will thus focus on homogeneous spaces $G/H$ where $H$ {\em is not} centralizer of any torus in $G$. Hermann has shown 
that in this case, with a few exceptions, $G/H$ does not carry any {\em invariant} almost complex structure (\cite[Thm. 5.3]{h55}). Of course,
his arguments being purely algebraic, he does not say anything about the possible existence of {\em non-invariant} almost complex
or, more generally, weakly complex structures. 

Let us start with some simple but important remarks on the behavior of 
weakly complex vector bundles on differentiable manifolds.

\begin{elem}\label{alt}
A real vector bundle $\tau$ on a compact manifold is weakly complex if and only if there exist a complex bundle $\gamma$
and a trivial bundle $\delta$ such that 
\be\label{tau}\tau\oplus\gamma=\delta.\ee
\end{elem}

\begin{proof}
Assume that $\tau$ is weakly complex, so there 
exist a complex bundle $\lambda$ and a trivial bundle $\epsilon$ such that $\tau\oplus\epsilon=\lambda$. Recall that for every complex vector bundle $\lambda$ on a compact manifold, there exists a complex vector bundle 
$\tilde\lambda$ such that $\lambda\oplus\tilde\lambda$ is trivial. We thus get
that $\tau\oplus(\tilde\lambda\oplus\epsilon)=\lambda\oplus\tilde\lambda$ is trivial. The relation
\eqref{tau} is thus satisfied for $\gamma:=\tilde\lambda\oplus\epsilon$ and $\delta:=\lambda\oplus\tilde\lambda$ if 
the rank of $\epsilon$ is even, and for $\gamma:=\tilde\lambda\oplus(\epsilon\oplus\RM)$ and $\delta:=(\lambda\oplus\tilde\lambda)\oplus\RM$ if 
the rank of $\epsilon$ is odd. The proof of the converse statement is similar.
\end{proof}

\begin{elem}\label{prod}
A product $M:=M_1\times M_2$ is weakly complex if and only if each factor is weakly complex.
\end{elem}
\begin{proof}
Let $p_i$ denote the standard projection $M\to M_i$.
If $M_i$ is weakly complex, there exists a trivial bundle $\epsilon_i$ over $M_i$ such that
$\T M_i\oplus\epsilon_i$ has a complex structure. Then $\tilde{\epsilon}_i:=p_i^*(\epsilon_i)$ are trivial bundles over
$M$ and $\T M\oplus\tilde{\epsilon}_1\oplus\tilde{\epsilon}_2=p_1^*(\T M_1\oplus\epsilon_1)\oplus p_2^*(\T M_2\oplus\epsilon_2)$
is a complex bundle.

Conversely, if there exists a trivial bundle $\epsilon$ such that $\T M\oplus\epsilon$ is complex, then 
the restriction of $\T M\oplus\epsilon$ to $M_1\times\{m_2\}$ is a complex bundle for each $m_2\in M_2$. On the other hand,
this restriction is stably isomorphic to $\T M_1$ since $\T M|_{M_1\times\{m_2\}}$ is the direct sum
of $\T M_1$ and a trivial bundle of rank $\dim(M_2)$. Thus $M_1$ is weakly complex, and similarly,
$M_2$ is weakly complex too.
\end{proof}

We now state two results, which basically say that if 
the total space of a homogeneous fibration carries an invariant almost
complex structure or has weakly complex tangent bundle, then the same holds for the fibers.

\begin{elem}\cite[Prop. 5.3]{h55}\label{lh}
Let $G$ be a compact connected Lie group, $L$ be a connected subgroup containing a maximal torus of $G$, contained in a chain of subgroups 
$L \subset L' \subset \ldots \subset  L^m \subset G$, 
with $L$ a maximal subgroup of $L', \ldots, L^m$ a maximal subgroup of $G$. If $G/L$ admits an invariant almost complex structure, so does $L'/L$.
\end{elem}

\begin{elem}\label{l2}
Let $G$ be a compact Lie group with closed subgroups $H$ and $H'$, such that
$H \subset H' \subset G$. If the total space of the fibration $G/H \rightarrow G/H'$ 
is weakly complex, then the same holds for the fiber $H'/H$.
\end{elem}
\begin{proof}
One can decompose the Lie algebras $\gg$ and $\hh'$ as $\gg=\hh'\oplus\mm$ and 
$\hh'=\hh\oplus\mm'$, where $\mm$ and $\mm'$ are the orthogonal complements of 
$\hh'$ in $\gg$ and $\hh$ in $\hh'$ with respect to some 
$\ad_{H'}$-invariant scalar product on $\gg$. The tangent bundle of the fiber $H'/H$ is associated to $H'$ via
the isotropy representation of $H$ on $\mm'$ and the tangent bundle of the total space $G/H$ is associated to $G$ via
the isotropy representation of $H$ on $\mm\oplus \mm'$. The restriction of $\T(G/H)$ to the fiber $H'/H$ is 
thus the direct sum of $\T(H'/H)$ and the bundle associated to $H'$ via the representation $\ad_H$ on $\mm$. This
representation tautologically extends to a $H'$ representation, thus showing that the normal bundle 
of the fiber is trivial. This implies  that the restriction of $\T(G/H)$ to $H'/H$ and $\T(H'/H)$ are stably
isomorphic, so $\T(H'/H)$ is weakly complex.
\end{proof}

Note that this result is valid for all locally trivial fibrations, since
the normal bundle of each fiber of a locally trivial fibration is trivial. 
The elementary proof above just avoids using the classical fact 
(see {\em e.g.} \cite[16.14.9]{d}) that $G/H \rightarrow G/H'$ is a locally trivial fibration.

As a partial converse to the above results, we describe two instances where 
the total space of a homogeneous fibration carries (weakly)
complex structures.

\begin{elem}\label{fib}
Let $G$ be a compact Lie group and let $H$ and $H'$ be two closed subgroups of $G$, such that
$H \subset H' \subset G$. If the base $G/H'$ and the fiber $H'/H$ of the fibration $G/H \rightarrow G/H'$ 
carry invariant almost complex structures, then the total space $G/H$ carries an invariant
almost complex structure too.
\end{elem}
\begin{proof}
As above we can write $\gg=\hh'\oplus\mm=\hh\oplus\mm'\oplus\mm$. The hypothesis ensures the existence
of an $\ad_H$-invariant complex structure on $\mm'$ and of an $\ad_{H'}$-invariant complex structure on $\mm$.
Their direct sum thus defines an $\ad_H$-invariant complex structure on $\mm\oplus \mm'$.
\end{proof}

\begin{elem}\label{stack} Let $H_k\subset H_{k-1}\subset\ldots\subset H_1\subset G$ be a sequence of
embeddings of closed subgroups of a compact Lie group $G$. If $G/H_1$ has a weakly complex tangent bundle and
$H_i/H_{i+1}$ has an invariant almost complex structure for every $1\le i\le k-1$, then $G/H_k$ has a weakly complex tangent bundle.
\end{elem}
\begin{proof}
By induction, it is clearly enough to prove the case $k=2$. As before, we decompose 
the Lie algebra $\gg=\hh_1\oplus\mm=\hh_2\oplus\mm'\oplus\mm$. By assumption, $\mm'$ has an $\ad_{H_2}$-invariant
complex structure $J'$.
The tangent bundle of $G/H_2$ decomposes as
$$\T(G/H_2)=(G\times_{\ad_{H_2}}\mm)\oplus (G\times_{\ad_{H_2}}\mm').$$
The first bundle is just the pull-back to $G/H_2$ of $\T(G/H_1)$, and is thus weakly complex, whereas the second bundle 
clearly has an invariant complex structure induced by $J'$. This proves the lemma.
\end{proof}

\section{Weakly complex inner symmetric spaces}

For the convenience of the reader we recall 
our previous classification results of weakly complex quaternion-K\"ahler manifolds and inner symmetric spaces.

\begin{ath}\cite[Th. 1.1]{gms} \label{t11}
Let $M^{4n}$, $n\ge 2$, be a compact quaternion-K\"ahler manifold of
positive type, which is not isometric to the complex Grassmannian
$\Gr_2(\CM^{n+2})$. Then the tangent bundle $\T M$ is not weakly complex.
\end{ath}

\begin{ath}\cite[Th. 1.3]{gms} \label{t13}
An irreducible component of a simply connected inner 
symmetric space of compact type admitting a
weak almost complex structure is isomorphic to an even-dimensional
sphere, or to a Hermitian symmetric space or (conceivably) to the
exceptional symmetric space
${\rm E} _7/({\rm SU} (8)/\mathbb{Z} _2)$. 
\end{ath} 

Using the methods developed in this paper we are now in position to rule out the exceptional case  $\E_7/\A_7$ in the above theorem. 

\begin{ath}\label{e7}
The tangent bundle of the exceptional symmetric space $\E_7/\A_7$ is not weakly complex.
\end{ath}

\begin{proof}
Assume for a contradiction that $\E_7/\A_7$ is weakly complex and consider the sequence of embeddings 
$$\A_3\times \A_3\times \T^1={\rm S}(\U(4)\times\U(4))\subset \SU(8)=\A_7\subset \E_7.$$
Since the fiber $\A_7/(\A_3\times \A_3\times \T^1)$ is Hermitian symmetric, the total space $\E_7/(\A_3\times \A_3\times \T^1)$ 
would be weakly complex by Lemma \ref{stack}. If $G$ denotes the centralizer
of the center $\T^1$ of $\A_3\times \A_3\times \T^1$ in $\E_7$, the fiber $G/(\A_3\times \A_3\times \T^1)$ of the fibration 
$\E_7/(\A_3\times \A_3\times \T^1)\to \E_7/G$ would be weakly complex by  Lemma \ref{l2}.

On the other hand, we claim that $G$ is isomorphic to $\D_6\times \T^1$ and that
the embedding $\A_3\times \A_3\times \T^1$ in $G$ is just the standard embedding of $\D_3\times \D_3\times \T^1$
in $\D_6\times \T^1$. This would then imply that the weakly complex manifold $G/(\A_3\times \A_3\times \T^1)$
is actually the real Grassmannian $\D_6/\D_3\times \D_3$ of 6-planes in $\RM^{12}$, 
thus contradicting Theorem \ref{t13}.

In order to prove our claim we need to study more carefully the embedding $\A_7\subset \E_7$
via the root systems. Recall first \cite{a96} that the root system of $\E_8$ is the disjoint union of the root
system of $\Spin(16)$ and the weights of the half-spin representation
$\Sigma^+_{16}$. It thus consists of the vectors $\pm e_i\pm e_j$,
$1\le i< j\le 8$ and
$$\frac 12 \sum_{i=1}^8 \e_i e_i,\qquad \e_i=\pm 1,\ \e_1\cdots
\e_8=1.$$
The vectors $\{e_i\}$ form an orthonormal basis of the maximal torus
$\RM^8$ of $\E_8$ with respect to some bi-invariant scalar product on the Lie algebra $\mathfrak{e}_8$
induced by the Killing form.
The root system of $\E_7$ is given by the set of roots of
$\E_8$ orthogonal to a fixed one, {\em e.g.} to $\a_0:=\tfrac12(e_1+\ldots+e_8)$:
$$\RR(\E_7)=\{\a\in\RR(\E_8)\ |\ \langle\a,\a_0\rangle=0\}.$$
The subset $\{e_i-e_j\ |\ 1\le i< j\le 8\}\subset \RR(\E_7)$ determines
the embedding $\A_7\subset\E_7$.

The roots of the subgroup $G\subset \E_7$ are those 
orthogonal to $e_1+e_2+e_3+e_4-e_5-e_6-e_7-e_8$, {\em i.e.}
$$\pm(e_i-e_j)\qquad\hbox{for}\ 1\le i<j\le 4\ \hbox{or}\ 5\le i<j\le 8$$
and
$$\frac 12 \sum_{i=1}^8 \e_i e_i,\qquad \e_i=\pm 1,\ \sum_{i=1}^4\e_i=\sum_{i=5}^8\e_i=0.
$$
The above system is isometric to the root system $\{\pm f_i\pm f_j,\ 1\le i<j\le 6\}$ of $\D_6\times \T^1$ by
defining
$$f_1=\frac12(e_1+e_2-e_3-e_4), \qquad f_4=\frac12(e_5+e_6-e_7-e_8),$$
$$f_2=\frac12(e_1-e_2+e_3-e_4), \qquad f_5=\frac12(e_5-e_6+e_7-e_8),$$
$$f_3=\frac12(e_1-e_2-e_3+e_4), \qquad f_6=\frac12(e_5-e_6-e_7+e_8),$$
$$f_7=e_1+e_2+e_3+e_4-e_5-e_6-e_7-e_8.$$
Moreover, this identification maps the roots 
$$\pm(e_i-e_j)\qquad\hbox{for}\ 1\le i<j\le 4\ \hbox{or}\ 5\le i<j\le 8$$
of $\A_3\times \A_3\times \T^1$ 
onto the roots 
$$\pm f_i\pm f_j\qquad\hbox{for}\ 1\le i<j\le 3\ \hbox{or}\ 4\le i<j\le 6$$
of $\D_3\times \D_3\times \T^1\subset\D_6\times \T^1$. This proves our claim and concludes the proof of the theorem.
\end{proof}

\section{The classification in the semi-simple case}

We now come back to the classification of weakly complex homogeneous spaces.
As a direct corollary of Lemma \ref{l1} and Lemma \ref{prod} we have:

\begin{epr}\label{p2}
Let $M = G/H$ be a compact simply connected homogeneous space with $\rk(G) = \rk(H)$. 
Then $M$ is weakly complex if and only if it is the product of 
homogeneous spaces $M_i$ with $M_i = G_i / H_i$, 
such that $M_i$ is weakly complex and each $G_i$ is a compact
simple Lie group.
\end{epr}

This proposition shows that the study of weakly complex equal rank homogeneous spaces $G/H$
reduces to the case where $G$ is simple. The first main step consists in the case where $H$ is semi-simple. The answer is provided by

\begin{ath}\label{t1}
Let $M = G/H$ be a simply connected equal rank compact homogeneous space such that $G$ is simple and
$H$ is semi-simple. Then $M$ is  weakly complex if and only if one of the following (exclusive) possibilities occurs:
\begin{enumerate}

\item[1.] \begin{enumerate} \item[a)]  $M$ is one of the seven spaces in the list \eqref{borel} of Proposition \ref{p1}.
\item[b)] $M=\E_8/(\A_5\times\A_2 \times \A_1)$.
\item[c)]  $M=\E_8/(\A_2)^4$.
\end{enumerate}

\item[2.]  \begin{enumerate}\item[a)]  $M=\SM^{2n}=\B_n/\D_n$ for $n\ge 2$.
\item[b)]  $M=\C_n/(\C_1)^n$ for $n\ge 3$.
\item[c)]  $M=\F_4/\D_4$.
\end{enumerate}
\end{enumerate}
Conversely, the spaces in 1. carry invariant almost complex structures and those in 2. 
have stably trivial (and thus weakly complex) tangent bundle but do not carry any invariant almost complex structure.
\end{ath}

\begin{proof}
If $H$ is maximal in $G$, Proposition \ref{p1} shows that either we are in case 1.a), 
or $M$ is an irreducible inner symmetric space.
In the latter situation, Theorem \ref{t13} together with Theorem \ref{e7} imply that either $M$ is an even dimensional sphere, so we are in case 2.a), 
or it is Hermitian symmetric (which is impossible since $H$ is semi-simple).

We thus may assume from now on that $H$ is not maximal in $G$. Let $H_1\subset G$ 
be a maximal connected closed subgroup of $G$ containing $H$. By an obvious inductive procedure one can construct
a sequence $H := H_{k} \subset H_{k-1} \subset \ldots \subset H_1 \subset H_{0}:= G $ ($k\ge 2$) of connected closed subgroups of $G$,
such that $H_{i+1}$ is maximal in $H_{i}$ for $0\le i\le k-1$.
Since $G/H$ fibers over $G/H_i$ with fiber $H_i/H$, Lemma \ref{l2} shows that $H_i/H$ is weakly complex for all $i$.
Moreover, since $H$ is semi-simple, Corollary \ref{c21} shows that the groups $H_i$ are semi-simple for all $i$.

On the other hand, Proposition~\ref{p1} shows that  $G/H_1$ either belongs to the
list \eqref{borel} of Proposition \ref{p1}, or is an irreducible inner symmetric space. 

{\bf Case 1:} $G/H_1$ belongs to list \eqref{borel}. By Lemma \ref{l3}, among the seven spaces in that list,
the only one which might occur is $G/H_1=\E_8/(\E_6\times \A_2)$, and $H_2=K\times \A_2$ for some maximal subgroup 
$K\subset \E_6$ of rank 6. By Proposition~\ref{p1} again, $H_1/H_2=\E_6/K$ is either inner symmetric or belongs
to the list \eqref{borel}. 

If $\E_6/K$ is inner symmetric, using the classification of symmetric spaces (\cite[pp. 312--314]{besse})
we get $K=\A_5\times \A_1$, so $H_2=\A_2\times\A_5\times \A_1$ 
and by Lemma \ref{l3} we must have $k=2$ {\em i.e.} $H_2=H$. On the other hand 
$H_1/H=\E_6/(\A_5\times \A_1)$ is a quaternion-K\"ahler symmetric space which is 
not weakly complex by  Theorem \ref{t11}, thus contradicting Lemma \ref{l2}.

If $\E_6/K$ belongs
to the list \eqref{borel},
the only possibility is $K=(\A_2)^3$, so $H_2=(\A_2)^4$ and  applying Lemma \ref{l3} again we see that $k=2$, {\em i.e.} $H_2=H$.
This shows that $M$ is the space in case 1.c).

{\bf Case 2:} $G/H_1$ is an irreducible inner symmetric space. 
Going through the list of these spaces (\cite[pp. 312--314]{besse}), 
and keeping in mind that $H_1$ is semi-simple, we distinguish several possibilities:

I. $G=\A_n$. This case is impossible by Lemma \ref{l3}.

II. $G=\B_n$. By Lemma \ref{class1} b), there exist integers $m\ge 0$, $p_i\ge 2$ for $1\le i\le k$, with $m+p_1+\ldots+p_k=n$
such that $H=\B_{m}\times\D_{p_1}\times\ldots\times \D_{p_k}$, diagonally embedded in $\B_n$. 
Since $H$ is not maximal in $G$, we either have $m\ge 1$, $k\ge 1$ or $m=0$, $k\ge 2$. If $m\ge 1$, 
the inclusion $H\subset H':=\B_{m+p_1}\times \D_{p_2}\times\ldots\times \D_{p_k}\subset G$ induces a fibration of $G/H$ 
over $G/H'$ with fiber $H'/H=\B_{m+p_1}/(\B_{m}\times\D_{p_1})$.
By Lemma \ref{l2}, the real Grassmannian $H'/H=\B_{m+p_1}/(\B_{m}\times \D_{p_1})$ 
has to be weakly complex, contradicting  Theorem \ref{t13}. If $m=0$, let $H'$ 
be the subgroup $\D_{p_1+p_2}\times\D_{p_3}\times\ldots\times \D_{p_k}$ of $G$ containing $H$. 
By Lemma \ref{l2} again, the real Grassmannian $H'/H=\D_{p_1+p_2}/(\D_{p_1}\times \D_{p_2})$ 
has to be weakly complex, contradicting  Theorem \ref{t13}. 

III. $G=\C_n$. By Lemma \ref{class1} c), there exist $k\ge 2$ integers $p_i\ge 1$ with $p_1+\ldots+p_k=n$
such that $H$ is conjugate to $\C_{p_1}\times\ldots\times \C_{p_k}$, diagonally embedded in $\C_n$. Since 
$H$ is not maximal in $\C_n$, we must have $k\ge 3$.
Assume that one of the $p_i$'s is larger than $1$ (say $p_1\ge 2$ for simplicity). 
The inclusion $H\subset H':=\C_{p_1+p_2}\times \C_{p_3}\times\ldots\times \C_{p_k}\subset G$ induces a fibration of $G/H$ 
over $G/H'$ with fiber $H'/H=\C_{p_1+p_2}/(\C_{p_1}\times\C_{p_2})$. By Lemma \ref{l2},
the quaternionic Grassmannian $H'/H=\C_{p_1+p_2}/(\C_{p_1}\times\C_{p_2)}$ 
has to be weakly complex, contradicting  Theorem \ref{t13} which says, 
in particular, that the only weakly complex quaternionic Grassmannian is the sphere $\SM^4=\C_2/(\C_1\times\C_1)$.
Thus $p_i=1$ for all $i$, and we are in case 2.b). 

IV. $G=\D_n$. By Lemma \ref{class1} d), 
there exist $k\ge 2$ integers $p_i\ge 2$ with $p_1+\ldots+p_k=n$
such that (up to conjugation) $H=\D_{p_1}\times\ldots\times \D_{p_k}$, diagonally embedded in $\D_n$. 
Like before, let $H'$ be the subgroup $\D_{p_1+p_2}\times\D_{p_3}\times\ldots\times \D_{p_k}$ of $G$ containing $H$. 
By Lemma \ref{l2}, the real Grassmannian $H'/H=\D_{p_1+p_2}/(\D_{p_1}\times \D_{p_2})$ 
has to be weakly complex, contradicting Theorem \ref{t13} again. 

V. $G=\G_2$ and $H_1=\A_1\times\A_1$. This case is impossible since by Lemma \ref{l3}, $H$, which is
a proper subgroup of $H_1$, cannot be semi-simple.

VI. $G=\F_4$ and $H_1=\C_3\times \A_1$. By Lemma \ref{l1}, every proper semi-simple rank 4 
subgroup of $H_1$ is of the form $K\times \A_1$
with $K\subset \C_3$. By Lemma \ref{l2}, $H_1/H=\C_3/K$ has to be weakly complex. Like in III. above, the only possibility 
is $K=(\C_1)^3=(\A_1)^3$, so $M=\F_4/(\A_1)^4$. In order to understand the embedding
$(\A_1)^4\subset\C_3\times \A_1\subset\F_4$, recall \cite{a96} that the root system of $\F_4$ 
is the disjoint union of the root
system of $\B_4$ and the weights of the spin representation
$\Sigma_{9}$. It thus consists of the vectors $\pm e_i\pm e_j$, 
$1\le i< j\le 4$, $\pm e_i$, $1\le i\le 4$ and
$$\frac 12 \sum_{i=1}^4 \e_i e_i,\qquad \e_i=\pm 1.$$
The embedding $\C_3\times \A_1\subset\F_4$ is determined by the subset of roots
$$\{\pm e_1,\pm e_2,\pm e_1\pm e_2, \pm(e_3+e_4), \frac12(\pm e_1\pm e_2\pm(e_3+e_4))\}\cup\{\pm(e_3-e_4)\}\subset\RR(\F_4).
$$
Indeed, the first subset on the left is isometric to the root system of $\C_3$ 
$$\RR(\C_3)=\{\pm f_i\pm f_j, \pm2f_i\}_{1\le i\le 3}$$
by taking $f_1=\frac12(e_1+e_2),\ f_2=\frac12(e_1-e_2),\ f_3=\frac12(e_3+e_4)$. On the other hand, the embedding 
$(A_1)^3\subset \C_3$ corresponds to the subset of roots $\{\pm 2f_i\}_{1\le i\le 3}$ of $\RR(\C_3)$, so finally 
the embedding of $(\A_1)^4$ into $\F_4$ corresponds to the subset of roots $\{\pm e_1\pm e_2,\pm e_3\pm e_4\}$.
Consequently, $(\A_1)^4=(\D_2)^2$ can also be embedded in $\F_4$ through the sequence of inclusions 
$$(\A_1)^4=\D_2^2\subset \D_4\subset\B_4\subset \F_4.$$
If $M=\F_4/(\A_1)^4$ were weakly complex, the same would hold by Lemma \ref{l2} for the real Grassmannian 
$\widetilde{\Gr}_4(\RM^8)=\D_4/(\D_2\times\D_2)$, which would contradict  Theorem \ref{t13}.

VII. $G=\F_4$ and $H_1=\B_4$. Since $H_1/H=\B_4/H$ is weakly complex, 
the argument in II. shows that the only possibility is $H=\D_4$, so $M=\F_4/\D_4$ is the space in case 2.c).

VIII. $G=\E_6$ and $H_1=\A_5\times\A_1$. By Lemma \ref{l3} once again, $H$ cannot be semi-simple.

IX. $G=\E_7$ and $H_1=\A_7$ or $H_1=\D_6\times A_1$. The first case is excluded by Lemma \ref{l3}
and in the second case, we obtain like in case IV. above that $H_1/H$ is not weakly complex. 

X. $G=\E_8$ and $H_1=\D_8$ or $H_1=\E_7\times \A_1$. In the first case, the argument in case IV. above shows that 
$H_1/H$ cannot be weakly complex. In the second case, by Proposition \ref{p1}, $H_2$ is one of the three groups
$\A_7\times\A_1$, $\A_5\times\A_2\times\A_1$, or $\D_6\times \A_1\times \A_1$. 
For the first two of these groups, Lemma \ref{l3} implies that $H=H_2$. If $H=\A_7\times\A_1$, Lemma \ref{l2} implies that the quotient
$H_1/H=\E_7/\A_7$ is weakly complex, contradicting Theorem \ref{e7}. If $H=\A_5\times\A_2\times\A_1$, we are in case 1.b). 
Finally, if $H_2=\D_6\times \A_1\times \A_1$ we cannot have
$H=H_2$, since then the fiber $H_1/H_2=\E_7/(\D_6\times\A_1)$ would be a compact quaternion-K\"ahler
manifold which is not weakly complex by Theorem \ref{t11}. Thus $H$ is a proper subgroup of $H_2$ and 
the argument in case IV. combined with Lemma \ref{l1} show that this case is impossible either.

For the converse statement, we recall that
the spaces in 1.a) carry an invariant almost complex structure by \cite[p. 500]{bh58}.
By Lemma \ref{fib}, the same holds for the two spaces in 1.b) and 1.c) because of the fibrations
$$\E_7/(\A_5\times \A_2)\hookrightarrow\E_8/(\A_5\times\A_2 \times \A_1)\to \E_8/(\E_7\times \A_1)$$ and 
$$\E_6/(\A_2)^3 \hookrightarrow\E_8/(\E_6\times\A_2)\to \E_8/(\A_2)^4$$ 
whose bases and fibers all
carry invariant almost complex structures (see also \cite[Thm. 5.3]{h55}). 

The spaces in 2.a)--2.c) are all weakly complex since their tangent bundle is stably trivial (see \cite[p. 159]{sw86}). 
The fact that they do not carry invariant almost complex structures follows from Hermann's classification \cite[Thm. 5.3]{h55}, 
however one can give a direct argument. Indeed, the spheres $\B_n/\D_n$ are symmetric spaces and
$D_n$ is semi-simple for $n\ge 2$, so \cite[Prop. 4.2]{h55} applies. For the remaining two cases one can use
\cite[Prop. 5.3]{h55} applied to the chains of subgroups $L\subset L'\subset G$
$$(\C_1)^n\subset\C_2\times (\C_1)^{n-2}\subset\C_n\qquad\hbox{and}\qquad \D_4\subset\B_4\subset\F_4,$$
for which $L'/L$ is either $\C_2/(\C_1)^2=\SM^4$ or $\B_4/\D_4=\SM^8$.
\end{proof}

\section{The classification in the non semi-simple case}

We now consider the case where $H$ is not semi-simple. Before stating the main result we need
to study in more detail some family of homogeneous spaces which will appear later on in the classification and
which require a different type of arguments.

Let $\HM^n$ denote the standard representation of $\C_n=\Sp(n)$. This representation has a quaternionic structure given by right multiplication with quaternions. We view $\HM^n$ as complex representation with respect to the right multiplication with $i$. The complex exterior power
$\Lambda^2\HM^n$ has a {\em real} structure defined by $r(x\otimes y):=xj\otimes yj$. Let $\llbracket\Lambda^2\HM^n\rrbracket$ denote the
real part of $\Lambda^2\HM^n$ with respect to $r$. It is generated by elements of the form $\llbracket x\otimes y\rrbracket:=x\otimes y+r(x\otimes y)$.

\begin{elem}\label{iso}
The restriction to $(\C_1)^n$ of $\llbracket\Lambda^2\HM^n\rrbracket$ is isomorphic, as real representation, to the direct sum
$\RM^n\oplus\mm_n$ between the trivial $n$-dimensional representation and the
isotropy representation of the manifold $\C_n/(\C_1)^n$.
\end{elem}

\begin{proof}
The restriction to  $(\C_1)^n$ of the standard $\C_n$ representation on $\HM^n$ decomposes as
$\HM^n = \oplus^n_{i=1} H_i$, where $H_i \cong  \HM^1$ denotes the $(\C_1)^n$ representation obtained
by composing the projection onto the $i$-th factor of $(\C_1)^n$ with the standard representation of
$\C_1$ on $\HM^1$. It is well known that the complexified Lie algebra of $\C_n$ can be identified with
$\mathrm{Sym}^2 \HM^n$, which as  $(\C_1)^n$ representation decomposes as
$$
\mathrm{Sym}^2 \HM^n =\bigoplus^n_{i=1} \mathrm{Sym}^2 H_i \, \oplus \, \bigoplus_{i<j}( H_i\otimes H_j ).
$$
It follows that the complexified isotropy representation $\mm_n\otimes\CM$ of the homogeneous space $\C_n/(\C_1)^n$
is just $ \oplus_{i<j} H_i\otimes H_j$.

Similarly we can decompose $\Lambda^2\HM^n$ and find
$
\Lambda^2\HM^n = \bigoplus^n_{i=1} \Lambda^2 H_i \oplus \bigoplus _{i<j} (H_i\otimes H_j).
$
The summands $ \Lambda^2 H_i $ are all one-dimensional and thus trivial $(\C_1)^n$ representations.
We thus obtain
$$
\Lambda^2\HM^n = \CM^n \oplus (\mm_n\otimes\CM) ,
$$
whence 
$
\llbracket\Lambda^2\HM^n\rrbracket = \RM^n \oplus \mm_n 
$ as claimed.
\end{proof}

\begin{ere}\label{triv} A corollary of this result is the fact, already mentioned before, that the homogeneous space 
$\C_n/(\C_1)^n$ is stably parallelizable.
Indeed, since any vector bundle on a homogeneous space $G/H$ associated to the restriction to $H$ of a $G$ representation is trivial, we have that
$$\RM^n\oplus\T (\C_n/(\C_1)^n)=\C_n\times_{(\C_1)^n}(\RM^n\oplus\mm_n)=\C_n\times_{(\C_1)^n}\llbracket\Lambda^2\HM^n\rrbracket$$
is a trivial vector bundle.
\end{ere}

Consider now the family of homogeneous spaces $G/H$ where 
$G=\C_n$ for some $n\ge 2$ and $H=(\C_1)^p\times U$ for some rank  $q$ subgroup
$U\subset \U(q)$ ($q=n-p$). The embedding of $H$ in $\C_n$ given by
$$(\C_1)^p\times U\subset (\C_1)^p\times \U(q)\subset (\C_1)^p\times\C_q\subset \C_n,$$
where $\U(q)\subset\C_q$ is the standard embedding obtained by viewing the complex entries of a matrix as quaternions,
and the last embedding is diagonal.

\begin{epr}\label{cn} The homogeneous spaces $\C_n/((\C_1)^p\times U)$ have weakly complex tangent bundle.
\end{epr}
\begin{proof} Let $T$ denote the center of $U$. It is clear that the centralizer of $T$ in $\C_n$ is $\C_p\times U$, so $\C_n/(\C_p\times U)$
has an invariant complex structure. In fact this space is a coadjoint orbit of $\C_n$ (see {\em e.g.} \cite[p. 234]{besse}). Its isotropy
representation, called $\nn$, is thus complex. Consider now the isotropy representation $\mm_p$ of the homogeneous space
$\C_p/(\C_1)^p$. Using the fibration $\C_n/((\C_1)^p\times U)\to \C_n/(\C_p\times U)$ we see that the isotropy
representation of the total space is the direct sum $\nn\oplus\mm_p$. Here $(\C_1)^p\times U$ acts by restriction from $\C_p\times U$ on
$\nn$ and the $U$ factor acts trivially on $\mm_p$. In order to finish the proof of the proposition we thus need to show that 
the associated bundle $\tau:=\C_n\times_{(\C_1)^p\times U}\mm_p$ is weakly complex.

Now, the restriction to $(\C_1)^p\times U$ of the representation $\HM^n$ of $\C_n$ decomposes as $\HM^n=\HM^q\oplus \HM^p$.
Correspondingly, the restriction to $(\C_1)^p\times U$ of $\llbracket\Lambda^2\HM^n\rrbracket$ decomposes as
\be\label{lam}\llbracket\Lambda^2\HM^n\rrbracket=\llbracket\Lambda^2\HM^q\rrbracket\oplus \llbracket\HM^q\otimes \HM^p\rrbracket\oplus \llbracket\Lambda^2\HM^p\rrbracket.\ee
The crucial observation is that the first two summands are {\em complex} representations of the group $(\C_1)^p\times U$. Indeed, the
{\em left} multiplication with $i$ on $\HM^q$ induces complex structures 
$$J\llbracket x\wedge x'\rrbracket:=\llbracket ix\wedge x'\rrbracket,\qquad J\llbracket x\otimes y\rrbracket:=\llbracket ix\otimes y\rrbracket$$
on $\llbracket\Lambda^2\HM^q\rrbracket$ and $\llbracket\HM^q\otimes \HM^p\rrbracket$ which are compatible with the $(\C_1)^p\times U$ action.
On the other hand, the $U$ factor in $(\C_1)^p\times U$ acts trivially on $\llbracket\Lambda^2\HM^p\rrbracket$, so by Lemma \ref{iso}, 
$\llbracket\Lambda^2\HM^p\rrbracket$ and $\RM^p\oplus \mm_p$ are isomorphic as $(\C_1)^p\times U$ representations. Since 
$\llbracket\Lambda^2\HM^n\rrbracket$ is a $\C_n$ representation, the associated vector bundle 
$\epsilon:=\C_n\times_{(\C_1)^p\times U}\llbracket\Lambda^2\HM^n\rrbracket$
is trivial. From \eqref{lam} we see that this trivial vector bundle can be written as the direct sum
$$\epsilon=(\C_n\times_{(\C_1)^p\times U}\llbracket\Lambda^2\HM^q\rrbracket)\oplus (\C_n\times_{(\C_1)^p\times U} \llbracket\HM^q\otimes \HM^p\rrbracket)
\oplus \RM^p\oplus\tau$$
of two complex vector bundles, a trivial rank $p$ vector bundle, and $\tau$. By taking the direct sum with a trivial real line bundle for $p$ odd, we see that there exists a complex bundle $\lambda$ and a trivial bundle
$\epsilon$ such that $\tau\oplus\lambda=\epsilon$. Thus $\tau$ is weakly complex by Lemma \ref{alt}.
\end{proof}

The remaining part of this section is devoted to the proof of the classification result in the non semi-simple case:

\begin{ath}\label{t2}
Let $M = G/H$ be a simply connected equal rank compact homogeneous space such that $G$ is simple and
$H$ is not semi-simple. Then $M$ is  weakly complex if and only if one of the following (exclusive) possibilities occurs:
\begin{enumerate}
\item[1.] \begin{enumerate} \item[a)]  $H$ is the centralizer of a torus in $G$.
\item[b)] $M$ is one of the four exceptional spaces $\E_8/(\A_5\times\A_2 \times \T^1)$, $\E_8/((\A_2)^3\times \A_1\times \T^1)$, 
$\E_8/((\A_2)^3\times\T^2)$ or $\E_7/((\A_2)^3\times \T^1)$.
\end{enumerate}
\item[2.]  \begin{enumerate}\item[a)]  $M=\B_n/(\D_{p}\times \U(q_1)\times\ldots\times \U(q_l))$, 
for some integers $p\ge 2$ and $q_i\ge 1$ for $1\le i\le l$ such that $n=p+q_1+\ldots+q_l$.
\item[b)]   $M=\C_n/((\C_{1})^p\times \U(q_1)\times\ldots\times \U(q_l))$, 
for some integers $p\ge 2$ and $q_i\ge 1$ for $1\le i\le l$ such that $n=p+q_1+\ldots+q_l$.
\item[c)]  $M$ is one of the exceptional spaces $\F_4/(\D_2\times \T^2)$ or $\F_4/((\C_1)^3\times\T^1)$.
\end{enumerate}
\end{enumerate}
The spaces in 1.~have invariant almost complex structures and those in 2.~have weakly complex tangent bundle, but do not carry any
invariant complex structure.
\end{ath}

\begin{proof}
Since $H$ is not semi-simple, its center $T$ is a toral subgroup of rank $r\ge 1$. We denote by $H'\supset H$ 
the centralizer of $T$ in $G$. Using the convention from Remark \ref{rem} 
we can assume that $H=K\times T$ and $H'=K'\times T$ with $K\subset K'$. 
The manifold $M=G/H$ fibers over $M':=G/H'$ with fiber
$K'/K$. Note now that $K$ and $K'$ are both semi-simple and have rank $n-r$. 
If $K'=K$, then $H$ is the centralizer of a torus in $G$, so $G/H$ has an invariant complex structure and we are in case 1.a).

Assume from now on that $K$ is a proper subgroup of $K'$. Since $K'$ is semi-simple
one can write $K'=K'_1\times\ldots\times K'_m$
with $K'_i$ simple for every $i$. By Lemma \ref{l1} we have $K=K_1\times\ldots\times K_m$
with $K_i:=K\cap K'_i$ semi-simple and $\rk(K_i)=\rk(K'_i)$ for every $i$.

Since $M$ is weakly complex, the same holds for $K'/K$ (by  Lemma \ref{l2}). Now, $K'/K$ is the direct product of
the spaces $K'_i/K_i$, and each factor $K'_i/K_i$ is weakly complex  by 
Lemma \ref{prod}. Consequently each factor $K'_i/K_i$ is
either a point, or one of the spaces given by Theorem \ref{t1}. By permuting the subscripts if necessary,
one can assume that $K_1$ is a proper subgroup of $K'_1$.

A useful observation is that by Corollary \ref{c1}, if $G$ is a classical compact simple Lie group 
 ({\em i.e.} in one of the series $\A$--$\D$), then a closed subgroup $K'\subset G$ with $\rk(K')=\rk(G)$
has no direct factor of exceptional type. 

In our present situation, this means that if 
$G$ is a classical group, then $K'_1$ is a simple classical group, thus $K'_1/K_1$ is one of the spaces in cases 
2.a) or 2.b) in Theorem \ref{t1}. In particular, we must have $K'_1=\B_{n'}$ or $K'_1=\C_{n'}$ for some $n'\ge 2$.
By Lemma \ref{class1}, $G=\B_n$ for some $n\ge 3$ in the first case or $G=\C_n$ for some $n\ge 3$ in the second case.

If $G=\B_n$, Lemma \ref{class1} also says that there exist integers $m,k\ge 0,\ l\ge 1$, $p_i\ge 1$ for $1\le i\le k$,  $q_i\ge 1$ for $1\le i\le l$,
with $m+p_1+\ldots+p_k+q_1+\ldots+q_l=n$
such that $H$ is conjugate to $\B_{m}\times\D_{p_1}\times\ldots\times \D_{p_k}\times \U(q_1)\times\ldots\times \U(q_l)$. By renaming the groups
$\D_1$ into $\U(1)$, we can assume $p_i\ge 2$ for all $i$. It is then easy to check that the centralizer of the center of
$H$ in $\B_n$ is $H'=\B_{p}\times \U(q_1)\times\ldots\times \U(q_l)$, where $p=m+p_1+\ldots+p_k$.
From Theorem \ref{t1}, $H'/H$ is weakly complex if only if $m=0$ and $k=1$, so we are in case 2.a). Conversely, if this holds,
{\em i.e.} $H$ is of the type $\D_{p}\times \U(q_1)\times\ldots\times \U(q_l)\subset\B_n$, then 
we can also embed $H$ in $\B_n$ as follows:
$$H=\D_{p}\times \U(q_1)\times\ldots\times \U(q_l)\subset\D_n\subset\B_n.$$
Lemma \ref{stack} then shows that $\B_n/H$ is weakly complex. Indeed, $\B_n/\D_n$ has stably trivial tangent bundle and
$\D_n/H$ is a coadjoint orbit \cite[p. 231]{besse} so it has an invariant complex structure. 

If $G=\C_n$, Lemma \ref{class1} implies that  there exist integers $k\ge 0$, $l\ge 1$, $p_i\ge 1$ for $1\le i\le k$,  $q_i\ge 1$ for $1\le i\le l$,
with $p_1+\ldots+p_k+q_1+\ldots+q_l=n$
such that $H$ is conjugate to $\C_{p_1}\times\ldots\times \C_{p_k}\times \U(q_1)\times\ldots\times \U(q_l)$. 
Like above, we check that  the centralizer of the center of
$H$ in $\C_n$ is $H'=\C_{p}\times \U(q_1)\times\ldots\times \U(q_l)$, where $p=p_1+\ldots+p_k$.
From Theorem \ref{t1}, $H'/H$ is weakly complex if only if $p_i=1$ for all $i$, {\em i.e.}
$H=(\C_{1})^p\times \U(q_1)\times\ldots\times \U(q_l)\subset \C_n$ where $n=p+q_1+\ldots+q_l$, so
we are in case 2.b). Conversely, each space in case 2.b) is weakly complex by Proposition \ref{cn}.

Consider now the case where $G$ is exceptional and $K'_1$ is classical. By Theorem \ref{t1},
$K'_1$ is either $\B_m$ or $\C_m$ for some $m\ge2$. Now, Proposition \ref{p1} together with the list of symmetric spaces
of exceptional type show that if $G$ is one of
$\G_2$, $\E_6$, $\E_7$ or $\E_8$ and $H'$ is a closed subgroup of $G$ with $\rk(H')=\rk(G)$, then $H'$ has
no factor isomorphic to $\B_m$ or $\C_m$ for $m\ge2$. We thus necessarily have $G=\F_4$. 
By Proposition \ref{p1} and \cite[pp. 312--314]{besse}, the maximal rank 4 proper subgroups of $\F_4$ are 
$\A_2\times \A_2$,  $\B_4$, and $\C_3\times \A_1$. Since $K_1'$ occurs as factor in one of their subgroups,
the first case can not occur. In the last two cases, using Lemma
\ref{class1} b) and c) several times we see that  $(H',H)$ necessarily belongs to the following list:
$(\B_3\times \T^1,\D_3\times \T^1)$, $(\B_2\times \U(2),\D_2\times \U(2))$, 
$(\B_2\times \T^2,\D_2\times \T^2)$, or $(\C_3\times\T^1,(\C_1)^3\times\T^1)$.
The first two candidates actually do not occur. Indeed, from the root system of $\F_4$ described in Theorem \ref{t1},
 we easily see that 
$\D_3\times \T^1$ is equal to the centralizer of its center in $\F_4$, and the centralizer in $\F_4$ of the center of $\D_2\times \U(2)$
is $\C_3\times\T^1$, which contains $\B_2\times \U(2)$ as proper subgroup.
In the last two cases $H'$ is indeed the centralizer in $F_4$ of the center of $H$ and moreover 
both groups $\D_2\times \T^2$ and
$(\C_1)^3\times\T^1$ embed in $\D_4$ as coadjoint orbits \cite[p. 230]{besse}.
Since $\F_4/\D_4$ is stably trivial \cite{sw86}, Lemma \ref{stack} shows that
the corresponding homogeneous spaces $\F_4/((\C_1)^3\times\T^1)$ and
$\F_4/(\D_2\times \T^2)$ are weakly complex.

It remains to treat the case where $K'_1$ and $G$ are simple exceptional groups. Using Lemma~\ref{l1},
Proposition \ref{p1} and the classification of symmetric spaces, we observe that 
a Lie algebra containing a summand isomorphic to $\mathfrak{g}_2$, $\mathfrak{f}_4$ 
or $\mathfrak{e}_8$ can not be properly embedded in a Lie algebra of the same rank. 
Looking at the different cases in Theorem \ref{t1}, we see that the only possibilities 
for $(K'_1,K_1)$ are $(\E_7,\A_2\times \A_5)$ and $(\E_6,(\A_2)^3)$.

In the first case we get $G=\E_8$ and $H=\A_2\times \A_5\times \T^1\subset H':=\E_7\times \T^1\subset \E_8$. 
The resulting space $M=\E_8/(\A_2\times \A_5\times \T^1)$ has an invariant almost 
complex structure, as shown by Lemma~\ref{fib}
applied to the fibration of $M\to\E_8/(\E_7\times \T^1)$ with fiber $\E_7/(\A_2\times\A_5)$. Indeed, the base is complex
homogeneous, being the twistor space of the compact quaternion-K\"ahler manifold $\E_8/(\E_7\times \A_1)$ and
the fiber has an invariant almost complex structure by Proposition~\ref{p1}. This space was curiously overlooked in 
Hermann's classification \cite[Thm. 5.3]{h55}.

In the second case, we either have $G=\E_7$ and $H=(\A_2)^3\times \T^1\subset 
H':=\E_6\times \T^1\subset \E_7$, or $G=\E_8$, $H'$ is one of $\E_6\times \T^2$ 
or $\E_6\times \A_1\times \T^1$ and correspondingly $H$ is $(\A_2)^3\times \T^2$ 
or $(\A_2)^3\times \A_1\times \T^1$. In each case the resulting spaces 
have invariant almost complex structures by Lemma \ref{fib}. 

The fact that the spaces in 2. do not carry invariant almost complex structures follows from 
\cite[Prop. 5.3]{h55} applied to the chains of subgroups $L\subset L'\subset G$
 $$\D_{p}\times \U(q_1)\times\ldots\times \U(q_l)\subset \B_{p}\times \U(q_1)\times\ldots\times \U(q_l)\subset\B_n$$
$$(\C_{1})^p\times \U(q_1)\times\ldots\times \U(q_l)\subset\C_2\times(\C_{1})^{p-2}\times \U(q_1)\times\ldots\times \U(q_l)\subset\C_n$$
$$\D_2\times \T^2\subset\B_2\times \T^2\subset\F_4$$
$$ (\C_1)^3\times\T^1\subset \C_2\times\C_1\times \T^1\subset\F_4$$
for which $L'/L$ is either $\B_n/\D_n=\SM^{2n}$ or $\C_2/(\C_1)^2=\SM^4$.

This completes the proof of the theorem.
\end{proof}

\end{document}